\documentclass[paper=a4, fontsize=11pt]{scrartcl} % A4 paper and 11pt font size

\usepackage[T1]{fontenc} % Use 8-bit encoding that has 256 glyphs
\usepackage{fourier} % Use the Adobe Utopia font for the document - comment this line to return to the LaTeX default
\usepackage[english]{babel} % English language/hyphenation
\usepackage{amsmath,amsfonts,amsthm} % Math packages
\usepackage{amsmath, calligra, mathrsfs}
\usepackage{amssymb}
\usepackage{enumerate}
\usepackage[mathscr]{euscript}
\usepackage{verbatim}
\usepackage{xcolor}
\usepackage{mdframed} % documentation "The mdframed package" on https://www.ps.uni-saarland.de/Publications/documents/Aiken_2009_File.pdf
\usepackage{soulutf8}
\usepackage{stmaryrd}
\usepackage{tikz-cd}
\usepackage{hyperref}
\usepackage{lipsum} % Used for inserting dummy 'Lorem ipsum' text into the template

\usepackage{sectsty} % Allows customizing section commands
\allsectionsfont{\centering \normalfont\scshape} % Make all sections centered, the default font and small caps

\usepackage{fancyhdr} % Custom headers and footers
\usepackage{xurl}
\newcommand*{\sheafhom}{\mathcal{H}\kern -.5pt om}
\numberwithin{equation}{section} % Number equations within sections (i.e. 1.1, 1.2, 2.1, 2.2 instead of 1, 2, 3, 4)
\numberwithin{figure}{section} % Number figures within sections (i.e. 1.1, 1.2, 2.1, 2.2 instead of 1, 2, 3, 4)
\numberwithin{table}{section} % Number tables within sections (i.e. 1.1, 1.2, 2.1, 2.2 instead of 1, 2, 3, 4)

\newtheorem{thm}{Theorem}[section]
\newtheorem{cor}[thm]{Corollary}
\newtheorem{prop}[thm]{Proposition}

\newtheorem{lem}[thm]{Lemma}

\newtheorem{quest}[thm]{Question}
\theoremstyle{definition}

\newtheorem{exmp}[thm]{Example}

\theoremstyle{remark}
\newtheorem{rem}[thm]{Remark}

\DeclareMathOperator{\Var}{Var}

\DeclareMathOperator{\Bl}{Bl}

\DeclareMathOperator{\cha}{char}

\newcommand{\horrule}[1]{\rule{\linewidth}{#1}} % Create horizontal rule command with 1 argument of height

\title{	
	\normalfont \normalsize 
	\textsc{} \\ [25pt] % Your university, school and/or department name(s)
	\horrule{0.5pt} \\[0.4cm] % Thin top horizontal rule
	\huge Characterizing cubic hypersurfaces via projective geometry

	\horrule{2pt} \\[0.5cm] % Thick bottom horizontal rule
}

\author{Soohyun Park} % Your name

\date{\normalsize September 16, 2022} % Today's date or a custom date

\begin{document}
	
	\maketitle

	We use the cut and paste relation $[Y^{[2]}] = [Y][\mathbb{P}^m] + \mathbb{L}^2 [F(Y)]$ in $K_0(\Var_k)$ of Galkin--Shinder for cubic hypersurfaces arising from projective geometry to characterize cubic hypersurfaces of sufficiently high dimension under certain numerical or genericity conditions. Removing the conditions involving the middle Betti number from the numerical conditions used extends the possible $Y$ to cubic hypersurfaces, complete intersections of two quadric hypersurfaces, or complete intersections of two quartic hypersurfaces. The same method also gives a family of other cut and paste relations that can only possibly be satisfied by cubic hypersurfaces. 
	
	\section{Introduction}

	The main objective of this note is to provide a characterization of cubic hypersurfaces using projective geometry under certain numerical/genericity conditions. This characterization is based on satisfying a relation of Galkin--Shinder (Theorem 5.1 on p. 16 of \cite{GS}) in the Grothendieck ring of varieties $K_0(\Var_k)$ called the $Y-F(Y)$ relation.  Given a reduced cubic hypersurface $Y \subset \mathbb{P}^{m + 1}$ and its Fano variety of lines $F(Y) \subset \mathbb{G}(1, m + 1)$, it states that  \[ [Y^{[2]}] = [Y][\mathbb{P}^m] + \mathbb{L}^2 [F(Y)]  \text{ (equivalently $[Y^{(2)}] = (1 + \mathbb{L}^m) [Y] + \mathbb{L}^2 [F(Y)]$ ) } \]  in $K_0(\Var_k)$, where $Y^{[2]}$ is the Hilbert scheme of two points on $Y$ and $Y^{(2)}$ is the second symmetric product. The general idea is to pair distinct points of $Y$ parametrizing a line with the third pont of intersection and the line defined by the initial pair of points. Here, the terms involving $F(Y)$ come from removing instances where the lines involved are contained in $Y$. We assume that $k$ is an algebraically closed field of characteristic $0$. As it turns out, the method we use is not unique to this relation and we indicate other relations in $K_0(\Var_k)$ which can fill this role. Some examples in lower dimensions are given in the final section. \\

	We give some context connecting this characterization to the overall structure of $K_0(\Var_k)$. Most work on the general structure of $K_0(\Var_k)$ or specific subrings generated by certain types of varieties focus on the space of possible relations (e.g. whether $\mathbb{L}$ is a zerodivisor). Instead, we focus on the converse question of what varieties satisfy a given relation. While this can be approached using existing work on varieties known to satisfy Larsen--Lunts' cut and paste property (e.g. varieties containing finitely many rational curves in Theorem 6.3.7 on p. 142 of \cite{CNS} although known to be false in general) or the graded ring associated to $K_0(\Var_k)$ \cite{Kub} via birational equivalence classes, there does not seem to be much known outside of these settings. We start to explore this in the case of Galkin--Shinder's $Y-F(Y)$ relation (Theorem 5.1 on p. 16 of \cite{GS}). 
	
	\begin{quest}(Farb) \label{whatyfy} \\
		Given a closed subvariety $Y \subset \mathbb{P}^n$ of dimension $m \ge 1$, let $F(Y) \subset \mathbb{G}(1, n)$ be the Fano variety of lines parametrizing lines in $\mathbb{P}^n$ contained in $Y$. If $Y$ satisfies the $Y-F(Y)$ relation, is it a cubic hypersurface? 
	\end{quest}

	In Section \ref{numyfy}, we address this question given some numerical/concrete topological conditions on the variety (Theorem \ref{resyfy}) which can be relaxed if we assume Hartshorne's conjecture or the Debarre--de Jong conjecture (Remark \ref{conjconseq}). 
	
	\begin{thm} \label{resyfy}
		
		Let $k$ be a field of characteristic $0$ such that $\overline{k} = k$. Let $Y \subset \mathbb{P}^n$ be a nondegenerate, irreducible, smooth projective variety over $k$ of dimension $m$ and degree $d$ satisfying one of the following conditions:
		
		\begin{enumerate}[a.]
			\item $d \le \frac{n}{4}$, or
			
			\item $Y$ can be defined by $\le \frac{n}{2}$ equations of degree $\le \frac{n}{2}$.
		\end{enumerate}

		Suppose that $Y$ is \emph{not} a 1-dimensional family of quadrics or a $2$-dimensional family of (projective) $(m - 2)$-planes (each isomorphic to $\mathbb{P}^{m - 2}$). If $m \ge 7$, then the first condition implies the second.
		
		\begin{enumerate}
			\item $Y$ satisfies the $Y-F(Y)$ relation $[Y^{[2]}] = [Y][\mathbb{P}^m] + \mathbb{L}^2 [F(Y)]$ 
			
			\item $Y$ is a cubic hypersurface, the intersection of two quadric hypersurfaces, or the intersection of two quartic hypersurfaces. The middle Betti number $b_m$ gives additional constraints:
			
				\begin{itemize}
					\item If $b_m$ is exponential in $m$, $Y$ is either a cubic hypersurface or the intersection of two quartics.
					
					\item If $m + 5 \le b_m < 2 \cdot 3^m - 5$, then $Y$ is a cubic hypersurface.
				\end{itemize}
			
		  	If the middle Betti number $b_m$ is exponential in $m$, it must either be a cubic hypersurface or the intersection of two quartic hypersurfaces. 
		\end{enumerate}
		
	\end{thm}
	
	Note that $F(Y)$ is connected if $Y$ is a cubic hypersurface of dimension $\ge 3$ (p. 12 of \cite{GS}). \\
	
	\begin{rem} \label{conjconseq}
		Here are some comments on the assumptions of Theorem \ref{resyfy}.
		
		\begin{enumerate}
			\item If Hartshorne's conjecture (part 1 of Remark \ref{conjsharp}) holds in codimension 2, then the conditions on $Y$ can be weakened to $2d - 4 \le n$ and $n \ge 7$. Note that the uniruledness property already implies $d \le n$ if $Y$ is a hypersurface and $d_1 + d_2 \le n$ if $Y$ has codimension 2 and $d_1, d_2$ are the degrees of the hypersurfaces whose intersection is equal to $Y$. Another possible replacement of the conditions is to take $Y$ contained in \emph{some} hypersurface of degree $d$ such that $2d - 4 \le n$. \\

			These numerical conditions can be removed if we assume that the analogue of Debarre--de Jong conjecture (Conjecture 1.2 on p. 1 of \cite{BR}) for complete intersections holds. This would imply that \[ \dim F(Y) = 2n - d - (n - m) - 2 = n - d + m - 2, \] where $d = \sum_{i = 1}^{n - m} d_i$ is the sum of the degrees of the hypersurfaces. As with the original Debarre--de Jong conjecture, it is well-known to hold in the generic case (e.g. Proposition 2.1 on p. 2 of arXiv version of \cite{Can}). When $n - m = 2$, this implies that $\dim F(Y) = m + 2 - d + m - 2 = 2m - d$, which is equal to $2m - 4$ only if $d = d_1 + d_2 = 4$. The proof of Theorem \ref{resyfy} eliminates this case. \\
			
			\item Moving to very generic properties, there are many examples of varieties that are not necessarily complete intersections which cannot satisfy the $Y-F(Y)$ relation (Example \ref{vgen}). 
			
			\item If $Y$ is a variety with a connected Fano variety of lines $F(Y)$ satisfying the $Y-F(Y)$ relation (e.g. the case of cubic hypersurfaces), it is \emph{not} a $1$-dimensional family of quadrics or a $2$-dimensional family of (projective) $(m - 2)$-planes since such varieties satisfying the $Y-F(Y)$ relation cannot be connected. This is shown in Proposition \ref{yfyfamex}. \\
			
		\end{enumerate}

	\end{rem}

	\begin{rem}

		The uniqueness of cubic hypersurfaces as varieties satisfying the $Y-F(Y)$ relation also applies in the localization rather than $K_0(\Var_k)$. This is because the same logic applies whenever the desired relation among varieties in $K_0(\Var_k)$ forces $\dim F(Y) = 2m - 4$ (Theorem 2 on p. 207 -- 208 of \cite{R}).  \\
	\end{rem}
 
	The proof of this result uniquely characterizes cubic hypersurface among generic hypersurfaces of a given degree.
	
	\begin{cor} (Corollary \ref{gencubic})
		Suppose that $Y \subset \mathbb{P}^n$ is a hypersurface generic among those of its degree. If $Y$ satisfies the $Y-F(Y)$ relation, then $Y$ is a cubic hypersurface. 
	\end{cor}
	
	Genericity also plays a role in a different approach involving uniruledness properties in Section \ref{uniruledyfy} (Proposition \ref{gencodim}, Corollary \ref{nodeg}). Here are some examples of results on varieties satisfying the $Y-F(Y)$ relation:
	
	\begin{cor} (Corollary \ref{degbound}) \\
		Assume that $\overline{k} = k$ and $\cha k = 0$ as above and suppose that $Y \subset \mathbb{P}^n$ is a $d$-dimensional variety satisfying the $Y-F(Y)$ relation. Then, the variety $Y$ is \emph{not} contained in a general hypersurface of degree $r > 2n - 3$. \\
	\end{cor}

	\begin{cor} (Corollary \ref{nodeg}) \\
		Assume that $\overline{k} = k$ and $\cha k = 0$ as usual. If $Y \subset \mathbb{P}^n$ is a variety of dimension $\ge 2$ which is the complete intersection of $m \ge 2$ hypersurfaces $W_i$ which are generic among those of their degrees $r_i$ for each $i$, then it does \emph{not} satisfy the $Y-F(Y)$ relation. \\
	\end{cor}
	
	Along the way, we obtain restrictions on varieties which satsify relations that share some properties with the $Y-F(Y)$ relation (Proposition \ref{kodimnewext}, Corollary \ref{anarel}) and answer a question of Cadorel--Campana--Rousseau \cite{CCR} related to a connection between uniruledness and symmetric products of varieties (Example \ref{symsb}). The steps used give an alternative method of restricting varieties satsifying the $Y-F(Y)$ relation (Proposition \ref{gencodim}, Corollary\ref{nodeg}). We also note that the $Y-F(Y)$ relation does \emph{not} generate all polynomial relations between terms involved in the $Y-F(Y)$ relation (Remark \ref{incomprel}).

	\section*{Acknowledgements}
	
	I am very grateful to my advisor Benson Farb for the motivating problem and helpful suggestions throughout the project. Also, I would like to thank him for extensive comments on preliminary versions of this paper. In addition, I'd like to thank Renjie Lyu for pointing out a missing assumption in an earlier version and very helpful discussions related to Fano varieties of lines.
	
	\section{Varieties satisfying the $Y-F(Y)$ relation} \label{yfyvars}
	
	In this section, we give some numerical and genericity conditions under which the $Y-F(Y)$ relation uniquely characterizes cubic hypersurfaces. 
	
	\subsection{Using numerical conditions} \label{numyfy}
	
	Our first main result (Theorem \ref{resyfy}) is a characterization of when varieties satisfy the $Y-F(Y)$ relation under certain numerical conditions. Assuming that the degree is sufficiently small compared to the dimension of the projective space it is embedded in, we can show that only cubic hypersurfaces satsify the $Y-F(Y)$ relation if $Y$ is \emph{not} a 1-dimensional family of quadrics or a 2-dimensional family of $(m - 2)$-planes ($m = \dim Y$). It turns out that the only possible smooth projective varieties $Y \subset \mathbb{P}^n$ which can satisfy the $Y-F(Y)$ relation are those which have codimension $\le 2$. Before giving the full proof of Theorem \ref{resyfy}, we give an outline of the arguments used. \\
	
	Recall that the $Y-F(Y)$ relation is given by $[Y^{[2]}] = [Y][\mathbb{P}^m] + \mathbb{L}^2 [F(Y)]$ in $K_0(\Var_k)$. This can be rewritten as $[Y^{[2]}] - [Y][\mathbb{P}^m] = \mathbb{L}^2 [F(Y)]$ (Theorem 5.1 on p. 16 of \cite{GS}). Substituting in Poincar\'e polynomials to this rearranged relation, each side has terms of degree $\ge 4$ since each side is a multiple of $p_{\mathbb{A}^1}(t)^2 =  t^4$. In particular, this means that the coefficients of $t^2$ and $t$  is $0$ on each side and $b_1 = 0$. The first nonzero term is from $t^{4m - 4}$ since the terms of degree $4m, 4m - 1, 4m - 2$, and $4m - 3$ are all equal to $0$. Dividing by $t^4$, we find that $\deg p_{F(Y)}(t) = 4m - 8$ and $\dim F(Y) = 2m - 4$. Both of these (and most of the computation in general) are shown in Lemma \ref{reldim}. Work of Rogora \cite{R} then implies that $Y$ has codimension $\le 2$ if $n$ is sufficiently large. The theorem then follows from applying results (e.g. work on Hartshorne's conjecture) which imply that $Y$ is a complete intersection under the given conditions. \\
	
	The first observation we make before proving Theorem \ref{resyfy} is that taking $m \ge 1$ implies that $F(Y) \ne \emptyset$. 
	
	\begin{lem} \label{noemp}
		Suppose that $Y \subset \mathbb{P}^n$ is a connected $m$-dimensional variety satisfying the $Y-F(Y)$ relation. If $m \ge 1$, the Fano variety of lines $F(Y) \ne \emptyset$. 
	\end{lem}
	
	\begin{proof}
		Suppose that $F(Y) = \emptyset$. In order for $Y$ to satisfy the $Y-F(Y)$ relation, we would need to have $[Y^{(2)}] = (1 + \mathbb{L}^d)[Y]$, which means that $p_{Y^{(2)}}(t) = \frac{1}{2} p_Y(t)^2 + \frac{1}{2} p_Y(t^2) = (1 + t^{2d}) p_Y(t)$. \\
		
		Writing $p_Y(t) = t^{2m} + b_{2m - 1} t^{2m - 1} + \ldots + b_1 t + b_0$, this would imply that $b_1 = b_2 = \cdots  = b_{2d - 1} = 0$. We can prove this by induction. The coefficient of $t^2$ in $\frac{1}{2} p_Y(t)^2 + \frac{1}{2} p_Y(t^2)$ is $\frac{1}{2} ((b_1^2 + 2b_0 b_2) + b_1)$ and its coefficient in $(1 + t^{2m}) p_Y(t)$ is $b_2$. Since $b_0 = 1$, this means that $b_1^2 + 2b_2 + b_1 = 2b_2 \Rightarrow b_1^2 + b_1 = 0$. Since $b_i \ge 0$ for each $i$, this implies that $b_1 = 0$. Suppose that $b_1 = b_2 = \cdots = b_{k - 1} = 0$ for some $k \le 2m - 1$. Consider the coefficient of $t^{2k}$ on each side. This implies that $\frac{1}{2} (b_k^2 + 2 b_0 b_{2k} + 2 b_1 b_{2k - 1} + \ldots + 2 b_{k - 1} b_{k + 1} + b_k) = b_{2k}$. Since $b_0 = 1$ and $b_1 = b_2 = \cdots  = b_{k - 1} = 0$, this means that $b_k^2 + 2 b_{2k} + b_k = 2 b_{2k} \Rightarrow b_k^2 + b_k = 0$. Since $b_k \ge 0$, this implies that $b_k = 0$. Thus, we have that $b_1 = b_2 = \cdots  = b_{2m - 1} = 0$ and $p_Y(t) = t^{2m}$. \\
		
		Substituting this back into $p_{Y^{(2)}}(t) = \frac{1}{2} p_Y(t)^2 + \frac{1}{2} p_Y(t^2) = (1 + t^{2m}) p_Y(t)$ gives $t^{4m} + t^{4m} = 2(1 + t^{2m}) \cdot t^{2m}$, which is clearly false. 
	\end{proof}
	
	Afterwards, the initial reduction is in determining $\dim F(Y)$. This involves looking at degrees of terms in the Poincar\'e polynomials.

	\begin{lem} \label{reldim}
		If a variety $Y \subset \mathbb{P}^n$ in Theorem \ref{resyfy} satisfies the $Y-F(Y)$ relation, then $\dim F(Y) = 2m - 4$.
	\end{lem}
	
	\begin{proof}
		Before substituting in Poincar\'e polynomials, we will rewrite $[Y^{[2]}] = [Y][\mathbb{P}^m] + \mathbb{L}^2 [F(Y)]$ as $[Y^{[2]}] - [Y][\mathbb{P}^m] = \mathbb{L}^2 [F(Y)]$. Since $Y^{[2]}$ is the blowup of $Y^{(2)}$ along the diagonal $\Delta \cong Y$, we have that $[Y^{[2]}] = [Y^{(2)}] + ([\mathbb{P}^{m - 1}] - 1)[Y]$. Since $Y^{(2)}(t) = \frac{1}{2} p_Y(t)^2 + \frac{1}{2} p_Y(t^2)$ (p. 188 of \cite{Ch} with $x = y = t$), the fact that $p_{\mathbb{P}^{m - 1}}(t) = 1 + t^2 + \ldots + t^{2m - 2}$ and $p_{\mathbb{A}^1}(t) = t^2$ implies that	
		\begin{align*}
			p_{Y^{[2]}}(t) - p_Y(t) p_{\mathbb{P}^m}(t) &= p_Y^{(2)}(t) + (t^2 + \ldots + t^{2m - 2}) p_Y(t) - (1 + t^2 + \ldots + t^{2m}) p_Y(t) \\
			&= \frac{1}{2} p_Y(t)^2 + \frac{1}{2} p_Y(t^2) - (1 + t^{2m}) p_Y(t).
		\end{align*}
		
		If the $Y-F(Y)$ relation holds, then \[ \frac{1}{2} p_Y(t)^2 + \frac{1}{2} p_Y(t^2) - (1 + t^{2m}) p_Y(t) = t^4 p_{F(Y)}(t). \]
		
		Since each side is a multiple of $t^4$, this means that the $t^2$ term on each side is equal to $0$. Note that $b_0 = b_{2m} = 1$ since $Y$ is connected. On the left hand side, this means that 
		\begin{align*}
			\frac{1}{2}(b_0 b_2 + b_1^2 + b_2 b_0) + \frac{1}{2} b_1 - b_2 &= \frac{1}{2}(2b_0 b_2 + b_1^2) + \frac{1}{2} b_1 - b_2 \\
			&= \frac{1}{2} (2b_2 + b_1^2) + \frac{1}{2} b_1 - b_2 \\
			&= b_2 + \frac{b_1^2}{2} + \frac{1}{2} b_1 - b_2 \\
			&= \frac{b_1^2}{2} + \frac{1}{2} b_1 \\
			&= 0 \\
			\Rightarrow b_1 &= 0.
		\end{align*}
		
		Poincar\'e duality implies that $b_{2m - 1} = 0$. \\
		
		Also, we have that the coefficient of $t^4$ is nonzero since $F(Y) \ne \emptyset$ by Lemma \ref{noemp}. Then, the coefficient of $t^4$ (given by $b_0(F(Y))$) is equal to 
		\begin{align*}
			\frac{1}{2}(2b_0 b_4 + 2b_1 b_3 + b_2^2) + \frac{1}{2} b_2 - b_4 &= b_0 b_4 + b_1 b_3 + \frac{b_2^2}{2} + \frac{b_2}{2} - b_4 \\
			&= b_4 + 0 + \frac{b_2^2}{2} + \frac{b_2}{2} - b_4 \\
			&= \frac{b_2^2}{2} + \frac{b_2}{2} \\
			&\ne 0 \\
			\Rightarrow b_2 &\ne 0.
		\end{align*}
		
		We use this to look at the coefficients of $t^{4m}, t^{4m - 1}, t^{4m - 2}, t^{4m - 3}$, and $t^{4m - 4}$. Since $\frac{1}{2} p_Y(t^2)$ can only contribute to even degree terms, we first look at the odd degree terms. The coefficient of $t^{4m - 1}$ on the left hand side is $\frac{1}{2}(2b_{2m - 1} b_{2m}) - b_{2m - 1} = b_{2m - 1} - b_{2m - 1} = 0$ since $4m - 1 > 2m$ when $m \ge 4$ and $Y$ is connected. This means that the coefficient of $t^{4m - 1}$ is $0$. Similarly, the coefficient of $t^{4m - 3}$ is equal to $\frac{1}{2}(2b_{2m - 3} b_{2m} + 2b_{2m - 2} b_{2m - 1}) - b_{2m - 3} = b_{2m - 3} - b_{2m - 3} = 0$ since $b_{2m - 1} = b_1 = 0$. \\
		
		Next, we find the coefficients of $t^{4m}, t^{4m - 2}$, and $t^{4m - 4}$. The coefficient of $t^{4m}$ is $\frac{1}{2} b_{2m}^2 + \frac{1}{2} b_{2m} - b_{2m} = \frac{1}{2} + \frac{1}{2} - 1 = 0$. The coefficient of $t^{4m - 2}$ is  
		\begin{align*}
			\frac{1}{2}(2b_{2m - 2} b_{2m} + b_{2m - 1}^2) + \frac{1}{2} b_{2m - 1} - b_{2m - 2} &= b_{2m - 2} b_{2m} + \frac{1}{2} b_{2m - 1}^2 + \frac{1}{2} b_{2m - 1} - b_{2m - 2} \\ 
			&= b_{2m - 2} + 0 + 0 - b_{2m - 2} \\ 
			&= 0
		\end{align*}
		since $b_{2m} = 1$ and $b_{2m - 1} = b_1 = 0$. \\
		
		However, the coefficient of $t^{4m - 4}$ is nonzero since it is equal to 
		\begin{align*}
			\frac{1}{2}(2b_{2m - 4} b_{2m} + 2b_{2m - 3} b_{2m - 1} + b_{2m - 2}^2) + \frac{1}{2} b_{2m - 2} - b_{2m - 4} &= \frac{1}{2}(2b_{2m - 4} + b_{2m - 2}^2) + \frac{1}{2} b_{2m - 2} - b_{2m - 4} \\
			&= b_{2m - 4} + \frac{1}{2} b_{2m - 2}^2 + \frac{1}{2} b_{2m - 2} - b_{2m - 4} \\
			&= \frac{1}{2} b_{2m - 2}^2 + \frac{1}{2} b_{2m - 2}  \\
			&= \frac{1}{2} b_2^2 + \frac{1}{2} b_2 \\
			&\ne 0.
		\end{align*}
		
		This implies that $\deg t^4 p_{F(Y)}(t) = 4m - 4$ and $\deg p_{F(Y)} = 4m - 8 \Rightarrow \dim F(Y) = 2m - 4$ since $F(Y)$ is projective (and therefore compact). 
	\end{proof}

	The proof of Lemma \ref{reldim} and Remark \ref{conjconseq} imply the following:
	
	\begin{cor} \label{gencubic}
		Suppose that $Y \subset \mathbb{P}^n$ is a hypersurface generic among those of its degree. If $Y$ satisfies the $Y-F(Y)$ relation, then $Y$ is a cubic hypersurface. 
	\end{cor}
	
	Combining the work above with a result of Rogora \cite{R}, we now obtain a restriction on the codimension of $Y$ in $\mathbb{P}^n$ which gives Theorem \ref{resyfy} after combining this with results on Hartshorne's conjecture. 
	
	\begin{proof}(Proof of Theorem \ref{resyfy}) \\
		Since $\dim F(Y) = 2m - 4$ by Lemma \ref{reldim}, the initial conditions of Theorem  \ref{resyfy} and the following result of Rogora \cite{R} imply that the codimension of $Y$ in $\mathbb{P}^n$ is $\le 2$. \\
		
		\begin{thm} (Rogora, Theorem 2 on p. 207 -- 208 of \cite{R}, Theorem 2' on p. 209 of \cite{R}) \\
			Let $k \ge 4$ and $X \subset \mathbb{P}^n$ be a $k$-dimensional irreducible subvariety of a projective space of dimension $n$. Let $\Sigma \subset \mathbb{G}(1, n)$ be a component of maximal dimension of the variety of lines containe din $X$. If $\dim \Sigma = 2k - 4$, one of the following holds:
			
			\begin{enumerate}
				\item $X$ is a $1$-dimensional infinite family of quadrics
				
				\item $X$ is a $2$-dimensional infinite family of projective $(k - 2)$-planes
				
				\item $X$ is a linear section of $\mathbb{G}(1, 4)$
				
				\item $\dim X \ge n - 2$
			\end{enumerate}
			
		\end{thm}

		We split the remaining cases into hypersurfaces and codimension $2$ varieties. \\
		
		Suppose that $Y$ is a hypersurface. Since $Y^{[2]}$ and $Y$ are both smooth and projective, Larsen--Lunts' stable birational equivalence result implies that $Y$ must be uniruled (Theorem 6.1.5 on p. 134 of \cite{CNS}, Corollary 2.6.3 on p. 476 of \cite{CNS} with $Z = \mathbb{P}^r$). This implies that $\deg Y \le n$ if $Y$ is a hypersurface. In this case, a result of Beheshti--Riedl \cite{BR} implies that $\dim F(Y)$ is equal to the ``expected dimension'' $2(m + 1) - d - 3 = 2m - d - 1$ if $n \ge 2 \deg Y - 4$ (Theorem 1.3 of \cite{BR}): 
		
		\begin{thm} \label{brresult} (Beheshti--Riedl, Theorem 1.3 on p. 2 of \cite{BR}) \\
			Let $X \subset \mathbb{P}^n$ be a smooth hypersurface. Then $F_k(X)$ will be irreducible of the expected dimension provided that \[ n \ge 2 \binom{d + k - 1}{k} + k. \]  
			
			In the special case $k = 1$, we can improve the bound, proving that $F_1(X)$ is of the expected dimension $2n - d - 3$ if $n \ge 2d - 4$ and irreducible if $n \ge 2d - 1$ and $n \ge 4$.
		\end{thm}

		Since $n \ge 2$ and the ``expected dimension'' is equal to $2m - 4$ if and only if $d = 3$ , $Y$ must be a cubic hypersurface. \\
		
		It remains to consider the codimension $2$ case. If $Y$ is a complete intersection, the uniruledness restriction implies that $d_1 + d_2 \le n$ (Section 4.4 on p. 99 of \cite{D}). For example, $Y \subset \mathbb{P}^n$ is a complete intersection if it is defined scheme-theoretically by $\le \frac{n}{2}$ equations (p. 588 of \cite{BEL}, \cite{Fal}). Suppose that these equations have degree $\le \frac{n}{2}$ as in condition b in Theorem \ref{resyfy}. Let $Z_i$ be the hypersurface of degree $d_i$. Since $Y \subset \mathbb{P}^n$ has codimension $2$, we have that $n = m + 2$. Since $d_i \le \frac{n}{2}$, Theorem \ref{brresult} applies and $\dim F(Z_i) = 2(m + 2) - d_i - 3 = 2m - d_i + 1$ (i.e. the expected dimension). If $d_i \ge 6$ for \emph{either} of $i = 1, 2$, then $\dim F(Y_i) \le 2m - 5$ and $\dim F(Y) \le \dim F(Y_i) \le 2m - 5$. This would make it impossible for $Y$ to satisfy the $Y-F(Y)$ relation. Thus, a codimension $2$ complete intersection of hypersurfaces of degree $(d_1, d_2)$ can only satisfy the $Y-F(Y)$ relation only if $d_1, d_2 \le 5$. \\ 
		
		We can further narrow down the remaining degree cases using a result of Canning (Theorem 1.3 on p. 2128 of \cite{Can}), which implies that the Debarre--de Jong conjecture holds for Fano complete intersections $(d_1, d_2)$ with $d_1 + d_2 \le 7$. In other words, $F(Y)$ has the expected dimension $2n - d_1 - d_2 - 4$ if $d_1 + d_2 \le 7$. In order to have $\dim F(Y) = 2m - 4$, we need $d_1 + d_2 = 4$. Since we assumed that $Y$ is nondegenerate (i.e. not contained in a hyperplane), we have that $d_i \ge 2$ for $i = 1, 2$ and the only instance where $d_1 + d_2 = 4$ is $d_1 = d_2 = 2$. It remains to study the cases where $d_1 + d_2 \ge 8$ and $d_1, d_2 \le 5$. These are the pairs $(d_1, d_2) = (3, 5), (4, 4), (4, 5)$. \\
		
		Let $Z_i$ be a hypersurface of degree $d_i$ for $i = 1, 2$. Since $n \ge m \ge 7 > 2(5) - 4 = 6$, Theorem \ref{brresult} implies that Fano varieties of lines of hypersurfaces of degree $5$ in $\mathbb{P}^n$ have the expected dimension $2m - 4$. Suppose that $d_2 = 5$. If $F(Z_1) \cap F(Z_2)$ is a nontrivial intersection, $Y = Z_1 \cap Z_2$ does \emph{not} satisfy the $Y-F(Y)$ relation. We can show that this is indeed the case when $d_1 = 3, 4, 5$. If the intersection $F(Z_1) \cap F(Z_2)$ is trivial, then we either have $F(Z_1) \subset F(Z_2)$ or $F(Z_2) \subset F(Z_1)$. A dimension count implies that the only possibility is $F(Z_2) \subset F(Z_1)$. However, the fact that $Z_1$ and $Z_2$ are covered by lines implies that $Z_2 \subset Z_1$, which is impossible if $Z_1$ intersects $Z_2$ transversely. Thus, the only possible degrees $(d_1, d_2)$ of codimension 2 varieties satisfying the $Y-F(Y)$ relation are $(2, 2)$ and $(4, 4)$. \\

		When $\deg Y \ll n$, existing results on Hartshorne's conjecture actually force $Y$ to be a complete intersection. If the codimension of $Y$ in $\mathbb{P}^n$ is 2,  it suffices to take $\deg Y \le \frac{n}{4}$ in the codimension 2 case by the following result of Bertram--Ein--Lazarsfeld \cite{BEL}. The reasoning above on the codimension 2 case can then be repeated.
		
		\begin{thm} (Bertram--Ein--Lazarsfeld, Corollary 3 on p. 588 of \cite{BEL})
			Assume that $X \subset \mathbb{P}^r$ is a smooth variety of degree $d$, dimension $n$, and codimension $e$. If \[ d \le \frac{r}{2e} \left[ = \frac{n}{2e} + \frac{1}{2} \right], \] then $X$ is a complete intersection.
		\end{thm} 
		
		Finally, we consider additional restrictions among cubic hypersurfaces and codimension 2 complete intersections of degree $(2, 2)$ or $(4, 4)$ coming from the middle Betti number. Since $b_m$ is either equal to $m + 1$ or $m + 4$ if $Y$ is a complete intersection of two quadrics (p. 20 of \cite{Rei}), we can exclude this case if we assume that $b_m$ is exponential in $m$. We can also remove complete intersections of two quartics $(4, 4)$ if we assume that $b_m < 2 \cdot 3^m$. Let $H(a_1, \ldots, a_r)$ be the two-variable generating function for Hodge numbers of complete intersections of hypersurfaces of degree $(a_1, \ldots, a_r)$ in $\mathbb{P}^n$ (Th\'eor\`eme 2.3 on p. 52 and Corollaire 2.4 on p. 53 of \cite{Gro}, p. 19 of \cite{Rei}). Note that $H(a, a) = 2H(a) + (1 + y)(1 + z) H(a)$ by the reasoning on p. 20 of \cite{Rei} and the middle primitive Betti number of a smooth hypersurface of degree $a$ and dimension $u > 0$ is $\frac{(-1)^u}{a}(a - 1 + (1 - a)^{u + 2})$ (Corollary 1.8 on p. 14 of \cite{Hu}). Thus, we have that $b_m(Y) > 2 \cdot 3^m - 5$ if $Y$ is the complete intersection $(4, 4)$ of two quadrics.
		
	\end{proof}
	
	\begin{rem} \label{conjsharp} ~\\
		\vspace{-5mm}
		\begin{enumerate}
			\item Hartshorne's original conjecture states that a smooth variety $Y \subset \mathbb{P}^n$ of dimension $m$ is a complete intersection if $(n - m) < \frac{1}{3} n$ (p. 1017 of \cite{H3}). An overview of work related to this conjecture is given in the introduction to \cite{ESS}. If the original conjecture holds, then a codimension $2$ variety $Y$ is a complete intersection if $n \ge 7$. 
			
			\item A possible method of approaching the codimension 2 case of Theorem \ref{resyfy} using more explicit examples without using results on Hartshorne's conjecture is related to work of Lanteri--Palleschi \cite{LP} (Proposition 2.1 (j) on p. 863 of \cite{LP}, Theorem 2.1 on p. 153 -- 154 of \cite{BP}) according to the result of Rogora \cite{R} with the codimension restriction on varieties $Y \subset \mathbb{P}^n$ of dimension $m$ such that $\dim F(Y) = 2m - 4$ (Remark 3 on p. 208 of \cite{R}). 
		\end{enumerate}
	\end{rem}
	
	If we restricted to varieties with a \emph{connected} Fano variety of lines, we can exclude $1$-dimensional families of quadrics and $2$-dimesnional families of (projective) $(m - 2)$-planes from those varieties characterized by the $Y-F(Y)$ relation.
	
	\begin{prop} \label{yfyfamex}
		There are $1$-dimensional families of quadrics or $2$-dimensional families of (projective) $(m - 2)$-planes with connected Fano varieties of lines which do \emph{not} satisfy the $Y-F(Y)$ relation.
	\end{prop}
	
	\begin{proof}
		Suppose that $F(Y)$ is connected. If $Y$ is taken to be a $1$-dimensional family of quadrics or $2$-dimensional family of (projective) $(m - 2)$-planes, the $Y-F(Y)$ relation might not necessarily be satisfied. While checking this, we will use the version of the $Y-F(Y)$ relation which states that $[Y^{(2)}] = (1 + \mathbb{L}^m)[Y] + \mathbb{L}^2 [F(Y)]$. In the constructions below, we will use the fact that there linear subspaces contained in the the Grassmannian (with respect to the Pl\"ucker embedding) which can be constructed out of pencils of linear subspaces of $\mathbb{P}^n$ containing a fixed linear subspace $A$ and contained in a fixed linear subspace $A$ and containing a fixed linear subspace $B$ of $\mathbb{P}^n$ (Theorem 3.16 on p. 110 and Theorem 3.22 on p. 115 of \cite{HT}). Given a Grassmannian $\mathbb{G}(r, n)$, these can be used to construct linear subspaces of dimension $\le \max(n - r, r + 1)$. \\

		\begin{enumerate}
			
			\item Suppose that $2m \le n$ and consider the trivial $1$-dimensional family of quadrics $\mathbb{P}^1 \times Q$, where $Q \subset \mathbb{P}^m$ is a quadric hypersurface. Note that we take family to mean a (projective) line in the $\left( \binom{m + 2}{2} - 1 \right)$-dimensional projective space giving the Hilbert scheme of quadrics of dimension $m - 1$. By Example 2.4.7 on p. 73 -- 74 of \cite{CNS}, we have that $[Q] = [\mathbb{P}^{m - 1}]$ if $m$ is even and $[Q] = [\mathbb{P}^{m - 1}] + \mathbb{L}^{\frac{m - 1}{2}}$ if $m$ is odd. \\
			
			If $m$ is even, this implies that $[Y] = (\mathbb{L} + 1) (\mathbb{L}^{m - 1} +  \ldots + \mathbb{L} + 1) = \mathbb{L}^m + 2 \mathbb{L}^{m - 1} + \ldots + 2 \mathbb{L} + 1$ and  \[ (1 + \mathbb{L}^m)[Y] = 1 + 2 \mathbb{L} + \ldots + 2 \mathbb{L}^{2m - 1} + 2 \mathbb{L}^{2m}. \]
			
			Since $([X_1] + \ldots + [X_\ell])^{(n)} = \sum_{n_1 + \ldots + n_\ell = n} \prod_{i = 1}^\ell [X_i]^{(n_i)}$ (Remark 4.2 on p. 5 of \cite{G1}), we have that 
			
			\begin{align*}
				[Y^{(2)}] &= 1 + 3 \mathbb{L}^2 + \ldots + 3 \mathbb{L}^{2m - 2} + \mathbb{L}^{2m} + 3 \mathbb{L} + 3 \mathbb{L}^2 + \ldots + 3 \mathbb{L}^{2m - 2} \\ 
				&+ \mathbb{L}^{2m} + 4 \mathbb{L}^3 + \ldots + 4 \mathbb{L}^m + 2 \mathbb{L}^{m + 1} + \ldots + 2 \mathbb{L}^{2m - 1}
			\end{align*} 
			
			Since the quadratic term of $(1 + \mathbb{L}^m) [Y]$ is $2 \mathbb{L}^2$ and the quadratic term of $[Y^{(2)}]$ is $5 \mathbb{L}^2$, the quadratic term of $[Y^{(2)}] - (1 + \mathbb{L}^m) [Y]$ is $3 \mathbb{L}^2$. Since the $Y-F(Y)$ relation states that $[Y^{(2)}] = (1 + \mathbb{L}^m) [Y] + \mathbb{L}^2 [F(Y)] \Longrightarrow \mathbb{L}^2 [F(Y)] = [Y^{(2)}] - (1 + \mathbb{L}^m) [Y]$, we can apply the Poincar\'e polynomial motivic measure to see that this contradicts our assumption that $F(Y)$ is connected since the constant term is $3$ instead of $1$. \\
			
			Suppose that $m$ is odd. Then, we have that $[Q] = [\mathbb{P}^{m - 1}] + \mathbb{L}^{\frac{m - 1}{2}}$. This implies that $[Y] = 1 + 2 \mathbb{L}  + 2 \mathbb{L}^2 + \ldots + 2 \mathbb{L}^{\frac{m - 3}{2}} + 3 \mathbb{L}^{\frac{m - 1}{2}} + 3 \mathbb{L}^{\frac{m + 1}{2}} + 2 \mathbb{L}^{\frac{m + 3}{2}} + \ldots + 2 \mathbb{L}^{m - 1} + \mathbb{L}^m$ and \[ (1 + \mathbb{L}^m) [Y] = 1 + 2 \mathbb{L}  + \ldots + 2 \mathbb{L}^{\frac{m - 3}{2}} + 3 \mathbb{L}^{\frac{m - 1}{2}} + 3 \mathbb{L}^{\frac{m + 1}{2}} + 2 \mathbb{L}^{\frac{m + 3}{2}} + \ldots + 2 \mathbb{L}^{m - 1} + \mathbb{L}^m. \]	
			
			Using the same method as above, we have that 
			\begin{align*}
				[Y^{(2)}] &= 1 + 3 \mathbb{L}^2 + \ldots + 3 \mathbb{L}^{m - 3} + 4 \mathbb{L}^{m - 1} + 4 \mathbb{L}^{m + 1} + 3 \mathbb{L}^{m + 3} + \ldots + 3 \mathbb{L}^{m - 1} + \mathbb{L}^{2m} \\
				&+ 1 + 2 \mathbb{L} + 2 \mathbb{L}^2 + \ldots + 2 \mathbb{L}^{\frac{m - 3}{2}} + 3 \mathbb{L}^{\frac{m - 3}{2}} + 3 \mathbb{L}^{\frac{m - 1}{2}} + 2 \mathbb{L}^{\frac{m + 3}{2}} + \ldots + 2 \mathbb{L}^{m - 1} + \mathbb{L}^m + \ldots + 2 \mathbb{L}^{2m - 1}.
			\end{align*}
			
			As with the previous case, we find that the quadratic term of $(1 + \mathbb{L}^m) [Y]$ is $2 \mathbb{L}^2$ and the quadratic term of $[Y^{(2)}]$ is $5 \mathbb{L}^2$. If $\mathbb{P}^1 \times Q$ satisfies the $Y-F(Y)$ relation, we would find that the constant term of the Poincar\'e polynomial of $F(Y)$ is equal to $3$. This contradicts our assumption that $F(Y)$ is connected.

			\item As mentioned above, we consider $(m - 2)$-planes that we consider are $(m - 2)$-planes $\Gamma$ such that $A \subset \Gamma \subset B$ for some fixed (projective) $(m - 4)$-plane $A$ and $(m - 1)$-plane $B$ such that $A \subset B$. The family considered here is the union of such $(m - 2)$-planes. Using an inclusion-exclusion argument, we see that the class of this union of $(m - 2)$-planes in $K_0(\Var_k)$ is a polynomial in $\mathbb{L}$. Thus, it suffices to study point counts over $\mathbb{F}_q$. Since this is determined by quotients $\Gamma/A \subset B/A$ and $A$, we have that $\# Y (\mathbb{F}_q) = \# A(\mathbb{F}_q) \cdot \# (B/A)(\mathbb{F}_q)$. This means that $\# Y(\mathbb{F}_q) = (1 + q + \ldots + q^{m - 4}) (1 + q + q^2) = (1 + 2q + 3q^2 + \ldots)$ and the coefficient of $q^2$ is equal to $3$. On the other hand, the fact that $([X_1] + \ldots + [X_\ell])^{(n)} = \sum_{n_1 + \ldots + n_\ell = n} \prod_{i = 1}^\ell [X_i]^{(n_i)}$ (Remark 4.2 on p. 5 of \cite{G1}) implies that the $\# Y^{(2)}(\mathbb{F}_q) = 1 + 2q + 6q^2 + \ldots$.  Matching coefficients, this implies that it is impossible for $F(Y)$ to be connected.
			
		\end{enumerate}

		In general, it is difficult to check whether a particular relation in $K_0(\Var_k)$ is satisfied since the projections from the incidence correspondences used to define the families of varieties used here (p. 208 -- 209 of \cite{R}) are not necessarily Zariski locally trivial fibrations or piecewise trivial fibrations. \\ 
		
	\end{proof}
	
	The methods above have similar impliciationsfor other relations in $K_0(\Var_k)$. 
	
	\begin{cor}
		
		Consider a polynomial relation in $K_0(\Var_k)$ with integer coefficients involving symmetric products of an $m$-dimensional smooth projective variety $Y \subset \mathbb{P}^n$, its second symmetric product $Y^{(2)}$, the Fano variety of lines $F(Y) \subset \mathbb{G}(1, n)$ and $\mathbb{A}^1$. If it can only be satisfied by $Y$ such that $\dim F(Y) = 2m  - 4$,  the only varieties defined by $\le \frac{n}{2}$ equations or of degree $\le \frac{n}{4}$ that can satisfy the relation are cubic hypersurfaces.

	\end{cor}
	
	\begin{rem} \label{incomprel}
		While the $Y-F(Y)$ is close to characterizing cubic hypersurfaces, the $Y-F(Y)$ relation does \emph{not} necessarily generate possible polynomial relations between a cubic hypersurface $Y \subset \mathbb{P}^{m + 1}$ and its Fano variety of lines $F(Y) \subset \mathbb{G}(1, m + 1)$. This can be checked explicitly for the case $m = 2$ or $m = 3$ when $Y$ is taken to have an ordinary double point (Example 5.8 on p. 20 of \cite{GS}).
	\end{rem}
	
	\color{black}
	
	\subsection{Approach via uniruledness property and an application to a stable rationality question} \label{uniruledyfy}
	
	Our initial approach to Question \ref{whatyfy} was to use uniruledness properties of varieties which can satisfy the given relation. The most recent/general restriction using this approach (Proposition \ref{gencodim}, Corollary \ref{nodeg}) is given below. These methods can be generalized to analyze other relations in $K_0(\Var_k)$ with some properties similar to the $Y-F(Y)$ relation. More specifically, we will apply the logic of the proof of Proposition \ref{kodimnew} to related stable birationality problem (Example \ref{symsb}) and obtain geometric restrictions to varieties satisfying relations sharing certain properties with the $Y-F(Y)$ relation (Example \ref{kodimnewext}). \\

	Using a result on unirationality of symmetric products (Proposition \ref{kodimnew}), we will first show that complete intersections of $\ge 2$ general hypersurfaces (in their given degrees) of dimension $\ge 2$ do \emph{not} satisfy Galkin--Shinder's $Y-F(Y)$ relation (Corollary \ref{nodeg}). Afterwards, we will use the proof of Proposition \ref{kodimnew} to rule out the $Y-F(Y)$ for the intersection of an integral variety with a \emph{very general} hypersurface (Example \ref{vgen}). Before doing this, we will state the birational geometry results used.

	\begin{prop} \label{kodimnew}
		Suppose that $\overline{k} = k$ and $\cha k = 0$. If $Y$ is a smooth projective variety of dimension $d \ge 1$ such that $[Y^{(2)}] \equiv [Y] \pmod{\mathbb{L}}$, the Kodaira dimension of $Y$ is $\kappa(Y) = -\infty$. If $d \ge 2$, the variety $Y$ must be uniruled. \\
	\end{prop}
	
	\begin{rem}
		In general, it is known that uniruled varieties have Kodaira dimension $-\infty$ (Corollary 1.11 on p. 189 of \cite{Ko2}). However, the converse is only known in dimension $\le 3$ (Conjecture 1.12 on p. 189 of \cite{Ko2}). \\
	\end{rem}

	\begin{proof}
		The blowup at the diagonal $\Bl_{\Delta_Y}(Y \times Y) \longrightarrow Y \times Y$ induces the blowup $Y^{[2]} \longrightarrow Y^{(2)}$, which is also a blowup at the diagonal. Since $Y$ is smooth and projective, the variety $Y^{[2]}$ is also smooth and projective (Example 7.3.1 on p. 169 of \cite{F2}, Theorem 3.1 on p. 5 of \cite{Le}, Theorem 1.4 on p. 10 of \cite{Ko2}). \\
		
		Since $Y^{[2]}$ and $Y$ are smooth projective varieties, the isomorphism $K_0(\Var_k)/(\mathbb{L}) \longrightarrow \mathbb{Z}[SB]$ from Larsen--Lunts' motivic measure (Theorem 2.3 on p. 87 of \cite{LL}, Theorem 6.1.5 on p. 134 of \cite{CNS}) implies that $Y^{[2]}$ and $Y$ are stably birational to each other. Since we assumed that $d \ge 1$, we have that $\dim Y < \dim Y^{[2]} = 2 \dim Y$. Since $Y$ and $Y^{[2]}$ are \emph{not} birational, the following results imply that $Z$ is uniruled: \\
		
		\begin{lem}(Lemma 33.25.10 of \cite{St}) \\
			Let $k$ be a field. Let $X$ be a variety over $k$ which has a $k$-rational point $x$ such that $X$ is smooth at $x$. Then $X$ is geometrically integral over $k$. \\
		\end{lem}
		
		\begin{cor}(Corollary 2.6.4 on p. 477 of \cite{CNS}) \label{sbvsbir} \\
			Let $k$ be a field of characteristic $0$, and $X$ and $Y$ be stably birational integral $k$-varieties such that $X$ is \emph{not} uniruled and $\dim Y \le \dim X$. Then $X$ and $Y$ are birational. \\
		\end{cor}
		
		Thus, the variety $Y^{[2]}$ is uniruled and its Kodaira dimension is $\kappa(Y^{[2]}) = -\infty$. Since $Y^{[2]}$ is a desingularization of $Y^{(2)}$, we have that \[ -\infty = \kappa(Y^{[2]}) = \kappa(Y^{(2)}) = 2\kappa(Y) \text{ (Theorem 1 on p. 1369 of \cite{AA}).}  \]  This implies that $\kappa(Y) = -\infty$. \\
		
		Since any variety dominated by a uniruled variety is uniruled, the variety $Y^{(2)}$ is also uniruled. In fact, the following result implies that $Y$ itself is uniruled if $d \ge 2$: \\ 
		
		\begin{cor}(Cadorel--Campana--Rousseau, Corollary 4.2 on p. 9 -- 10 of \cite{CCR}) \\
			Suppose that $X$ is compact and K\"ahler. 
			
			\begin{enumerate}
				\item $X$ is rationally connected if and only if $X^{(m)}$ is for some $m \ge 1$. \\
				
				\item $X$ is uniruled if and only if $X^{(m)}$ for some $m \ge 1$, unless $X$ is a curve of genus $g > 0$ and $m > g$. In that case, $X^{(m)}$ is uniruled and $X$ is \emph{not} uniruled. 
			\end{enumerate}
			
		\end{cor}
	\end{proof}
	
	Before applying the methods used here to study varieties satsifying the $Y-F(Y)$ relation, we obtian some initial restrictions on degrees of generic hypersurfaces containing them. Recall from Lemma \ref{noemp} that $F(Y) \ne \emptyset$ if $Y \subset \mathbb{P}^n$ satsifies the $Y-F(Y)$ relation. Applying Theorem 4.3 on p. 266 of \cite{Ko2} gives a bound on (general) hypersurfaces containing a variety satisfying the $Y-F(Y)$ relation. \\
	
	\begin{cor} \label{degbound}
		Suppose that $Y \subset \mathbb{P}^n$ is a $d$-dimensional variety satisfying the $Y-F(Y)$ relation. Then, the variety $Y$ is \emph{not} contained in a general hypersurface of degree $r > 2n - 3$. \\
	\end{cor}
	
	\begin{proof}
		As mentioned above, this is an application of Lemma  \ref{noemp} and the fact that the Fano variety of lines for a general hypersurface of degree $r > 2n - 3$ is empty (Theorem 4.3 on p. 266 of \cite{Ko2}). \\
	\end{proof}

	 Under these restrictions, we can obtain restrictions on varieties $Y \subset \mathbb{P}^n$ satisfying the $Y-F(Y)$ relation which are intersections of generic hypersurfaces (among given degrees). This is done using uniruledness properties and repeated applications of the projective dimension theorem (Theorem 7.2 on p. 48 of \cite{H2}) and the dimension of $F(Y)$ for a generic hypersurface of dimension $r$ (Theorem 4.3 on p. 266 of \cite{Ko2}):

	\begin{prop} \label{gencodim} 
		Suppose that $Y \subset \mathbb{P}^n$ is a $m$-dimensional variety satisfying the $Y-F(Y)$ relation. 
		
		\begin{enumerate}
			\item If $Y$ is the intersection of $u$ \emph{general} hypersurfaces $W_i$ of degree $r_i \le 2n - 3$ for each $i$, we have that \[ (u + 1)(n - (R + 2)) + (R + 1) \le 2m - 4, \] where $R = \max(r_1, \ldots, r_u)$. 
			
			\item If $\dim Y \ge 2$ and $Y$ is the complete intersection of $m \ge 2$ \emph{general} hypersurfaces $W_i$ of degree $r_i \le 2n - 3$ for each $i$, then it does \emph{not} satisfy the $Y-F(Y)$ relation.
		\end{enumerate}

	\end{prop}

	\begin{rem} ~\\
		\vspace{-5mm}
		\begin{enumerate}
			\item The condition on the $r_i$ was inserted to keep $F(Y)$ from being empty (Corollary \ref{degbound}). 
			
			\item Hartshorne's conjecture \ref{conjsharp} implies that Part 2 can be rewritten as generic condition on varieties of a given degree when the codimension is small (proof of Lemma \ref{reldim}).
			
			\item The dimension bound in Part 1 uses Lemma \ref{reldim}. \\
		\end{enumerate}
	\end{rem}

	\begin{proof}
		Both parts use the same initial setup. By definition, the Fano scheme $F(Y)$ is the intersection of $F(W)$ for hypersurfaces $W$ containing $Y$ (p. 196 of \cite{EH}). Writing $Y = W_1 \cap \cdots W_u$,this means that $F(Y) = F(W_1) \cap \cdots \cap F(W_u)$. Then, induction on $u$ and repeatedly using the projective dimension theorem (Theorem 7.2 on p. 48 of \cite{H2}) implies that \[ \dim F(Y) \ge \dim F(W_1) + \ldots + \dim F(W_u) - (u - 1)(n - 1) \] since $F(W_i) \subset \mathbb{G}(1, n) \cong \mathbb{P}^{n - 1}$. \\
		
		We first prove the dimension inequality in Part 1. 
		
		\begin{enumerate}
			\item Let $R = \max(r_1, \ldots, r_n)$. Since the hypersurfaces $W_i \subset \mathbb{P}^n$ are taken to be general among those of degree $r_i$ in $\mathbb{P}^n$ and $\dim F(W_i) = 2n - 3 - r_i$ for such degree $r_i$ hypersurfaces (Theorem 4.3 on p. 266 of \cite{Ko1}), this means that 
			
			\begin{align*}
				\dim F(Y) &\ge \dim F(W_1) + \ldots + \dim F(W_u) - (u - 1)(n - 1) \\
				&= (2n - 3 - r_1) + \ldots + (2n - 3 - r_u) - (u - 1)(n - 1) \\
				&= un - 2u + n - 1 - (r_1 + \ldots + r_u) \\
				&\ge un - 2u + n - 1 - Ru \\
				&= un - (R + 2)u + n - 1 \\
				&= u(n - (R + 2)) + (n - (R + 2)) + (R + 1) \\
				&= (u + 1)(n - (R + 2)) + (R + 1).
			\end{align*}
			
			Since $\dim F(Y) = 2m  - 4$ by Lemma \ref{reldim}, this implies that \[ (u + 1)(n - (R + 2)) + (R + 1) \le 2m - 4. \]

			\vspace{2mm} 
			
			Next, we prove Part 2 by combining the methods above with a uniruledness property. 
			
			\item 
			Since the complete intersection $Y = W_1 \cap \cdots \cap W_u$ is uniruled by Proposition \ref{kodimnew}, we have that $r_1 + \ldots + r_u \le n$ (Section 4.4 on p. 99 of \cite{D}). \\
			
			Recall from the proof of Part 1 that $\dim F(Y) \ge \dim F(W_1) + \ldots + \dim F(W_u)$.
			
			Since the hypersurfaces $W_i \subset \mathbb{P}^n$ are taken to be general among those of degree $r_i$ in $\mathbb{P}^n$, a standard result on dimensions of Fano varieties of lines on generic hypersurfaces of a given degree (Theorem 4.3 on p. 266 of \cite{Ko2}) implies that 
			
			\begin{align*}
				\dim F(Y) &\ge \dim F(W_1) + \ldots + \dim F(W_u) - (u - 1)(n - 1) \\
				&= (2n - 3 - r_1) + \ldots + (2n - 3 - r_u) - (u - 1)(n - 1) \\
				&= 2un - 3u - (r_1 + \ldots + r_u) - (un - u - n + 1) \\
				&\ge 2un - 3u - n - (un - u - n + 1) \\
				&= u(n - 2) - 1 \\
				&\ge 2(n - 2) - 1 \\
				&= 2n - 5
			\end{align*}
			
			However, this is impossible since $\dim F(Y) = 2m - 4$ by Lemma \ref{reldim}. 
		\end{enumerate}

	\end{proof}

	Using Proposition \ref{kodimnew}, we can consider the case where $Y$ is a complete intersection of general hypersurfaces. \\

	\begin{cor} \label{nodeg}
		If $Y \subset \mathbb{P}^n$ is a variety of dimension $\ge 2$ which is the complete intersection of $m \ge 2$ \emph{general} hypersurfaces $W_i$ of degree $r_i$ for each $i$, then it does \emph{not} satisfy the $Y-F(Y)$ relation. \\
	\end{cor}
	
	\begin{proof}
		By Corollary \ref{degbound}, a variety $Y \subset \mathbb{P}^n$ satisfying the $Y-F(Y)$ is \emph{not} contained in a general hypersurface of degree $r > 2n - 3$. Combining this with Part 2 of Proposition \ref{gencodim}, we cover general hypersurfaces of arbitrary degrees.
	\end{proof}
	
	If we look at \emph{very general} hypersurfaces, we can use the logic of the uniruled restriction to rule out the $Y-F(Y)$ relation in possibly non-complete intersections.
	
	\begin{exmp}(\textbf{The $Y-F(Y)$ relation and possibly non-complete intersections}) \label{vgen} \\
		Recall that we used a result comparing stable birational and birational equivalence classes of varieties (Corollary  \ref{sbvsbir}) to show that smooth projective varieties of dimension $\ge 2$ which satisfy the $Y-F(Y)$ must be uniruled. After substituting in the Poincar\'e polynomials, this restriction was used to show that a complete intersection of general hypersurfaces (of their given degrees) does \emph{not} satisfy the $Y-F(Y)$ relation. However, the varieties that fail to satsify the $Y-F(Y)$ relation are not necessarily complete intersections. By Proposition 1 on p. 1 of \cite{Ko1}, the intersection of an integral $k$-variety $X \subset \mathbb{P}^N$ and a very general hypersurface $H \subset \mathbb{P}^N$ is \emph{not} uniruled. Thus, such an intersection $X \cap H$ of dimension $\ge 2$ does \emph{not} satisfy the $Y-F(Y)$ relation. \\
	\end{exmp}
	
	The comparison between stable birational and birational equivalence in Corollary \ref{sbvsbir} can also be used to study stable birationality questions. 
	
	\begin{exmp}(\textbf{Symmetric products and stable birationality}) \label{symsb} \\
		In Question 1 on p. 10 of \cite{CCR}, it is asked whether $X^{(m)}$ being unirational (resp. rational,  stably rational) for some $m \ge 2$ implies that $X$ is unirational (resp. rational,  stably rational).  Applying Corollary 2.6.4 on p. 477 of \cite{CNS} to a smooth projective resolution of $X^{(m)}$, we can show that this is false for stable birationality if $X$ has Kodaira dimension $\kappa(X) \ge 0$.  \\
	\end{exmp}
	
	A more ``straightforward'' application of the argument of Proposition \ref{kodimnew} from the previous section is to restrict varieties which satisfy other relations which share certain properties with the $Y-F(Y)$ relation. \\

	Suppose a smooth projective variety $Y$ of dimension $\ge 1$ satisfies $[Y^{(a)}] \equiv [Y^{(b)}] \pmod{\mathbb{L}}$ for some $a > b$ (e.g. $a = 2$ and $b = 1$ in the $Y-F(Y)$ relation). If $\dim Y \ge 2$, the same we can use Corollary 4.2 on p. 9 -- 10 of \cite{CCR} again to find that such a $Y$ must be uniruled. Here is a generalization of uniruledness arguments restricting varieties satisfying the $Y-F(Y)$ relation taking these observations into account: \\
	
	\begin{prop}(Restatement of Kodaira dimension/uniruled restriction) \label{kodimnewext} \\
		If $Y$ is a smooth projective variety of dimension $d \ge 1$ such that $[Y^{(a)}] \equiv [Y^{(b)}] \pmod{\mathbb{L}}$ for some $a > b \ge 1$, the Kodaira dimension of $Y$ is $\kappa(Y) = -\infty$. If $d \ge 2$, the variety $Y$ must be uniruled. \\
	\end{prop}

	Again, the restriction on the Kodaira dimension immediately implies the following: \\
	
	\begin{cor} \label{anarel}
		The following hold:
		
		\begin{enumerate}
			\item If $Y \subset \mathbb{P}^{d + 1}$ is a $d$-dimensional hypersurface of degree $r$ such that $[Y^{(a)}] \equiv [Y^{(b)}] \pmod{\mathbb{L}}$ in $K_0(\Var_k)$ for some $a > b \ge 1$, then $r \le d + 1$. \\
			
			\item If $Y$ is a smooth projective surface such that $[Y^{(a)}] \equiv [Y^{(b)}] \pmod{\mathbb{L}}$ in $K_0(\Var_k)$, then $Y$ must be a ruled surface (over a curve of genus $\ge 1$) or a rational surface. \\
		\end{enumerate}
	\end{cor}

	Splitting varieties by whether they are uniruled or not (or by Kodaira dimension being nonnegative or not) can also give a general perspective on restricting varieties which satisfy given relations in $K_0(\Var_k)$: \\

	\begin{rem}(Uniruled varieties and problematic birational automorphisms) \\
		Results of Kuber (Theorem 4.1 and Theorem 5.1 on p. 482 -- 485 of \cite{Kub}) imply that the structure of (the graded ring associated to) $K_0(\Var_k)$ essentially depends on birational equivalence classes of varieties when we consider classes in $K_0(\Var_k)$ where Larsen--Lunts' cut and paste conjecture holds. By results of Liu and Sebag, this includes varieties containing only finitely many rational curves (Theorem 6.3.7 on p. 142 of \cite{CNS}) and algebraic surfaces whose $1$-dimensional components are rational curves (Corollary 6.3.8 on p. 144 of \cite{CNS}). In some sense, this is the ``opposite'' of the uniruled varieties mentioned above since we were looking at when a pair of stably birational varieties are not actually birational to each other in Proposition \ref{kodimnewext}. \\

	\end{rem}

Department of Mathematics, University of Chicago \\
5734 S. University Ave, Office: E 128 \\ Chicago, IL 60637 \\
\textcolor{white}{text} \\
Email address: \href{mailto:shpg@uchicago.edu}{shpg@uchicago.edu} 

\end{document}